\documentclass[10pt,a4paper]{article}
\usepackage[utf8]{inputenc}
\usepackage{amsmath}
\usepackage{amsfonts}
\usepackage{amssymb}
\usepackage{amsthm}
\usepackage{mathtools}
\usepackage{mathrsfs}
\usepackage{cite}

\newtheorem{thm}{Theorem}
\newtheorem{lem}[thm]{Lemma}
\newtheorem{prop}[thm]{Proposition}
\newtheorem{coro}[thm]{Corollary}
\theoremstyle{definition}
\newtheorem{defn}[thm]{Definition}

\newtheorem{example}[thm]{Example}
\theoremstyle{remark}
\newtheorem{rmk}[thm]{Remark}

\DeclareMathOperator{\supp}{supp}

\DeclareMathOperator{\divergence}{div}
\DeclareMathOperator{\proj}{proj}

\newcommand{\id}{\mathrm{Id}}

\newcommand{\R}{\mathbb{R}}
\newcommand{\Z}{\mathbb{Z}}

\newcommand{\lebesgue}{\mathscr L}

\newcommand{\T}{\mathbb T}

\renewcommand{\geq}{\geqslant}
\renewcommand{\leq}{\leqslant}

\date{}
\title{Deformations of closed measures and variational characterization of measures invariant under the Euler-Lagrange flow}
\author{Rodolfo R\'ios-Zertuche
}

\begin{document}
\maketitle

\begin{abstract}
The set of closed (or holonomic) measures provides a useful setting for studying optimization problems because it contains all curves, while also enjoying good compactness and convexity properties.

We study the way to do variational calculus on the set of closed measures. Our main result is a full description of the distributions that arise as the derivatives of variations of such closed measures. We give examples of how this can be used to extract information about the critical closed measures. 

The condition of criticality with respect to variations leads, in certain circumstances, to the Euler-Lagrange equations. To understand when this happens, we characterize the closed measures that are invariant under the Euler-Lagrange flow. Our result implies Ricardo Ma\~n\'e's statement that all minimizers are invariant. \end{abstract}

\tableofcontents

\section{Introduction}
\label{sec:intro}
Let $M$ be a $C^\infty$ manifold and $I=[0,t_0]$ be a closed interval in $\R$ for some $t_0>0$.
To each $C^1$ curve $\gamma\colon I\to M$, we may associate a measure on $TM\times I$ by pushing forward the Lebesgue measure on $I$ through the application
\[t\mapsto(\gamma(t),\gamma'(t),t),\quad t\in I.\]
This gives a Radon measure $\mu_\gamma$ on $TM\times I$ that acts as follows on measurable functions $f\colon TM\times I\to\R$:
\begin{equation}\label{eq:mugamma}
 \langle \mu_\gamma,f\rangle=\int_0^{t_0}f(\gamma(t),\gamma'(t),t)\,dt.
\end{equation}

Thus, given a Lagrangian density $L\colon TM\times I\to\R$, the problem of finding a curve $\gamma$ that minimizes the action
\[
 A_L(\gamma)=\int_0^{t_0} L(\gamma(t),\gamma'(t),t)\,dt
\]
is equivalent to the problem of finding a curve $\gamma$ that minimizes $\langle \mu_\gamma,L\rangle$. As the set of $C^1$ curves does not have good compactness and convexity properties as are required for a simple proof of the existence of the minimizers, it is an age-old approach \cite{giusti}
to relax the problem slightly to include the weak* closure $\mathscr H$ of the convexification of the set of measures $\mu_\gamma$; this space $\mathscr H$ is known as the set of \emph{closed measures}, and will be defined precisely in Definition \ref{def:closedmeasures}.

Traditionally, the calculus of variations proceeds by looking at curves $\gamma$ and wiggling them slightly, a method from which it is deduced that ---under the right circumstances--- the minimizers must satisfy the \emph{Euler-Lagrange equations},
\begin{multline}\label{eq:EulerLagrange}
 \frac{\partial L}{\partial x}(\gamma(t),\gamma'(t),t)=\frac{d}{dt}\left[\frac{\partial L}{\partial v}(\gamma(t),\gamma'(t),t)\right]\\=\gamma'(t)\frac{\partial^2L}{\partial x\partial v}+\gamma''(t)\frac{\partial^2 L}{\partial v^2}+\frac{\partial^2 L}{\partial v\partial t},
\end{multline}
where $L=L(x,v,t)$, $x\in M$, $v\in T_xM$, $t\in I$. In this paper, we investigate the outcome of taking this general principle of the calculus of variations and applying it in the more general setting of the set of closed measures $\mathscr H$. 

\begin{defn}\label{def:generalvariations}
 A variation of a measure $\mu$ is a family of measures $(\mu_s)_{s\geq 0}\subset \mathscr H$ starting at $\mu=\mu_0$ that is weakly differentiable at $s=0$, in the sense that there exists a compactly-supported distribution $\eta\in \mathscr E'$ such that, for all $f\in C^\infty(TM\times I)$, 
 \[
  \left.\frac{d}{ds}\right|_{s=0}\langle \mu_s,f\rangle=\langle\eta,f\rangle.
 \]
 The distribution $\eta$ is the \emph{derivative} of the variation $(\mu_s)_{s\geq0}$, and it will be denoted by $d\mu_s/ds|_{s=0}=\eta$.
\end{defn}
We find the following characterization of all distributions $\eta$ that can arise in this way; for the proof see Section \ref{sec:variations}.

\begin{thm}\label{thm:generalvariationscharacterization}
 Let $\mu\in\mathscr H$ be a closed measure, and let $\eta$ be an element of the space $\mathscr E'$ of compactly-supported distributions. Then there exists a variation $(\mu_s)_{s\geq 0}$ with $\mu_0=\mu$ and $d\mu_s/ds|_{s=0}=\eta$ if, and only if, we have 
 \[\langle\eta,\varphi\rangle\geq 0\]
 for all smooth functions $\varphi\in C^\infty(TM\times I)$ such that $\int \varphi\,d\mu=0$ and such that there is a Lipschitz function $f\colon M\times I\to \R$ vanishing on $M\times\{0,t_0\}$ that is differentiable on $\proj_{M\times I}(\supp\mu)$, and at all points of $M\times I$ where $f$ is differentiable it satisfies $\varphi\geq df$, with equality $\varphi=df $ throughout $\supp\mu$.
\end{thm}

The usual variations of a curve $\gamma$ through ``wiggling'' indeed produce variations of the measure $\mu_\gamma$ associated as in equation \eqref{eq:mugamma}. \emph{Wiggling} amounts to flowing the curve with a smooth set of diffeomorphisms $\Phi_s\colon M\to M$, to produce a smooth set of curves $\gamma_s=\gamma\circ\Phi_s$ and the associated family of measures $\mu_s=\mu_{\gamma_s}$.
Example \ref{ex:horizontal} studies these variations. 

Then, in Example \ref{ex:vectorfieldvariations}, Corollary \ref{coro:vectorfieldvariations} gives a useful statement: \emph{given a vector field $F$ on $TM\times I$, a family $(\mu_s)_{s\in\R}\subset \mathscr H$ exists that flows $\mu$ in the direction $F$ if, and only if, $F$ ``preserves the closedness of $\mu$,%
''} in a sense that is made precise in the statement of the corollary, and which is fairly easy to check in practice.
This result is applied in Example \ref{ex:verticalvariations} and also in Section \ref{sec:horizontalvariations}, where we define and study \emph{velocity-dependent horizontal variations}, an important generalization of wiggling.
Upon proving Corollary \ref{coro:vectorfieldvariations}, we prove Lemma \ref{lem:thatextraregularity}, a result interesting in itself as it gives slightly better regularity for the differential $df$ of Theorem \ref{thm:mathertimedependent} than was proved earlier in \cite{myminimizable}.

As the other examples in Section \ref{sec:examplesvariations} show, there are variations of measures that are more general in the sense that they do not arise as simple generalizations of curve wiggling, and valuable information can be deduced from these. For instance, we find the existence of Hamilton-Jacobi type controls (Example \ref{ex:verticalvariations}), energy conservation (Example \ref{ex:transpositional}), and regularity (Example \ref{ex:mather}, itself interesting because of its being an infinitesimal analogy to Mather's construction) of the measures critical with respect to certain families of variations.

After studying the variations of closed measures in Section \ref{sec:variations}, Section \ref{sec:invariant} is devoted to the study of the circumstances that make the Euler-Lagrange equations \eqref{eq:EulerLagrange} valid. A way to do this would be to see whether the support of a measure contains curves that obey those equations; we discuss this briefly Appendix \ref{sec:eulerlagrange}, where we deduce the Euler-Lagrange equations from the existence of a critical subsolution of the Hamilton-Jacobi equation.

However, we feel that the most interesting way to study the validity of the Euler-Lagrange equations when dealing with closed measures is to consider the flow on $TM\times I$ induced by these equations: if $\partial^2L/\partial v^2$ is invertible, then we can solve \eqref{eq:EulerLagrange} for $\gamma''$ to get the vector field $\mathcal X$ followed by the curve $(\gamma(t),\gamma'(t),t)$, namely, 
\[(\gamma(t),\gamma'(t),t)'=\mathcal X(\gamma(t),\gamma'(t),t)\coloneqq\left(\gamma'(t),L_{vv}^{-1}(L_x-\gamma'L_{xv}-L_{vt}),1\right).\]
The flow of the vector field $\mathcal X$ is known as the \emph{Euler-Lagrange flow}, and in Theorem \ref{thm:invariance} we characterize all the closed measures $\mu\in\mathscr H$ that are invariant under this flow. 

We remark that invariant measures are not in general absolute minimizers of the action of $L$. A common example are the measures invariant under the geodesic flow. Also, the Hamilton-Jacobi equation does not govern them.

Before stating our the characterization of measures invariant under the Euler-Lagrange flow let us remark that it is variational, in the sense that we give a family $\Theta_\Delta(L)\subset \mathscr E'$ of distributions such that invariance under the Euler-Lagrange flow corresponds to $\Theta_\Delta(L)$-criticality, defined in Section \ref{sec:horizontalcriticality}. The elements of $\Theta_\Delta(L)$ are the derivatives of the exact horizontal variations, also defined in that section.

It turns out, however, that a purely variational characterization is impossible, as shown in Section \ref{sec:impossiblityexample}, so in Section \ref{sec:closedvelocities} we introduce a convex, dense subset $\mathscr H_\Delta$ of $\mathscr H$ within which such a characterization can be realized.
As shown in Proposition \ref{prop:closedvelocities}, the set $\mathscr H_\Delta$ consists of all closed measures $\mu$ for which an analogue of the second derivative exists. This is equivalent to these measures being closed on $TM\times I$ (rather than only on $M\times I$), so we  call them \emph{closed measures with closed velocities}. 
The set $\mathscr H_\Delta$ is convex and dense in $\mathscr H$, and it contains all possible invariant measures (Example \ref{ex:invariantisclosedvelocity}).

\begin{thm}\label{thm:generaleulerlagrangecharacterization}
 Let $\mu$ be a compactly-supported, positive, Radon measure on $TM\times I$. Assume that the support of $\mu$ is contained in an open set $U\subset TM\times I$ in which the second derivative $\partial^2L/\partial v^2$ of the $C^2$ Lagrangian function $L\colon TM\times I\to\R$ is invertible. 
 Then the following are equivalent:
 \begin{itemize}
  \item The measure $\mu$ is invariant under the Euler-Lagrange flow.
  \item The measure $\mu$ is a closed measure with closed velocities (i.e., $\mu\in\mathscr H_\Delta$) and $\mu$ is horizontally critical, i.e., for all exact horizontal variations $(\mu_s)_{s}$ of $\mu=\mu_0$ we have
  \[\left.\frac{d}{ds}\int L\,d\mu_s\right|_{s=0}=0.\]
 \end{itemize}
 Please refer to Section \ref{sec:invariant} for precise definitions of invariance, horizontal criticality, and exact horizontal variations.
\end{thm}

In Section \ref{sec:insufficiency}, we show that the class of variations most akin to ``wiggling,'' which are defined in Example \ref{ex:horizontal}, which is contained in the class of exact horizontal variations (the one that appears in the theorem), does not suffice for a complete variational characterization of invariance under the Euler-Lagrange flow. We do not know, however, whether the class of exact horizontal variations is the smallest one possible.

It is a theorem of Ricardo Ma\~n\'e that all minimizers of so-called Tonelli Lagrangians are invariant under the Euler-Lagrange flow; see \cite[\S2-4]{contrerasiturriagabook}. A slightly more general result is a corollary of our invariance characterization; see Section \ref{sec:manethm}.

\begin{rmk}
 The rest of the paper is written taking $M$ to be an open set $U\subseteq\R^n$. We have chosen to do this for simplicity, as this does not affect the generality of our statements, which are all essentially of local nature and the versions on $\R^n$ amount to taking a chart on $M$ and working on local coordinates. The statements can be easily adapted to general $C^\infty$ manifolds $M$.
\end{rmk}

\paragraph{Acknowledgements.}
I am deeply grateful for the help of  Alberto Abbondandolo, who suggested the question of variationally characterizing Euler-Lagrange invariance, and shared with me his substantial previous work in this direction.

It took the author quite some time to get this paper right. Along the way, many more people contributed advice and insightful discussions, among which the following are remembered with special gratitude: Antonio Ache, Luigi Ambrosio,  Camilo Arias Abad, Marie-Claude Arnaud, Victor Bangert, Patrick Bernard, Jaime Bustillo, Gonzalo Contreras, Albert Fathi, Charles Hadfield, Renato Iturriaga, Rafael de la Llave, John N. Mather, Jes\'us Puente Arrubarrena, Mar\'ia Amelia Salazar, and Stefan Suhr.

I am also very grateful to Princeton University, to the Institute for Computational and Experimental Research in Mathematics at Brown University, to the Max Planck Institute for Mathematics in Bonn, to the \'Ecole Normale Superieure de Paris, and to the Universit\'e de Paris -- Dauphine for their hospitality and support during the development of this research. 

\section{Variations of closed measures}
\label{sec:variations}

\subsection{Setting}

Let $U$ be an open subset of $\R^n$, and let $I=[0,t_0]$ be a time interval. We will denote by $x=(x_1,x_2,\dots,x_n)$ the coordinates in $U$, and by $t$ the coordinate in $I$. We will denote by $T\R^n\cong\R^n\times\R^n$ the tangent bundle of $\R^n$, and by $\pi_{\R^n\times I}\colon T(\R^n\times I)\to\R^n\times I$ the canonical projection. The tangent bundle $TU\cong U\times \R^n$ of $U$ is the a subset of $T\R^n$. On the fibers $T_x\R^n\cong\R^n$ of the tangent bundle $T\R^n$, we set coordinates $v=(v_1,v_2,\dots,v_n)$. Denote by $\mathbf 1=\partial/\partial t$ the vector field in the tangent bundle $TI$ such that $dt(\mathbf 1)=1$ throughout $I$. Thus coordinates on $T\R^n\times I$ will be of the form $(x,v,t)=(x_1,\dots,x_n,v_1,\dots,v_n,t)\in\R^{2n+1}$.

Let $\mathscr E'$ be the space of compactly-supported distributions on $TU\times I$. The topology in $\mathscr E'$ is the weak* topology with respect to the space of smooth functions $C^\infty(TU\times I)$.

\begin{defn}[Closed measures]\label{def:closedmeasures}
Let $\mathscr H\subset \mathscr E'$ be the set of Radon, com\-pact\-ly-supported measures $\mu$ on $T\R^n\times I$ such that, for all $\phi$ in $C^\infty(\R^n\times I)$ supported in $U\times (0,t_0)$,
\begin{equation}\label{eq:closedness}
 \int_{T\R^n\times I}[\phi_x(x,t)\cdot v+\phi_t(x,t)\cdot\mathbf 1]\,d\mu(x,v,t)=0.
\end{equation}
In other words, the boundary of the current induced by the measures in $\mathscr H$ lies outside of $U\times(0,t_0)$.
We will say that the elements of $\mathscr H$ are \emph{measures closed in $U$}, and when it is clear what $U$ is, we will simply refer to them as \emph{closed measures}.
\end{defn}

The vector $\mathbf 1=\partial/\partial t$ appears because, in the time $t$ direction, our objects have unit velocity.

\begin{rmk}\label{rmk:approximability}
 The equivalence between this definition and our description of $\mathscr H$ in the introduction follows from results in \cite{federerrealflat} (also explained in \cite[Theorem 1.3.4.6 in Part \textsc{ii}]{giaquintamodicasoucek2}; see also \cite{patrick,bangert}) that ascertain that every measure satisfying \eqref{eq:closedness} can be approximated by convex combinations of measures of the form $\mu_\gamma$. The intuition is that this should be true because, by the Fundamental Theorem of Calculus, all measures $\mu_\gamma$ satisfy \eqref{eq:closedness}, namely, $\langle\mu_\gamma,d\phi\rangle=\phi(\gamma(t_0),t_0)-\phi(\gamma(0),0)=0$ for all curves $\gamma$ and all $\phi$ as in the definition, so all measures in the closure of the convexification of the set of measures $\mu_\gamma$ should satisfy it too. 
\end{rmk}

\begin{defn}\label{def:variationRn}
Given $\mu\in \mathscr H$ and $\eta\in \mathscr E'$, we say that \emph{$\mu$ can be deformed in direction $\eta$} if there is a family of measures $(\mu_t)_{t\geq 0}$ contained in $\mathscr H$ such that $\mu_0=\mu$ and, for all $\varphi\in C^\infty(T\R^n\times I)$ supported in $TU\times I$,
 \[\left.\frac{d}{ds}\right|_{s=0+}\int_{T\R^n\times I}\varphi\,d\mu_s=\langle\eta,\varphi\rangle.\]
In this case the family $(\mu_t)_{t\geq 0}$ is a \emph{variation} of $\mu$.
\end{defn}

The terminology in the following definition is motivated by Theorem \ref{thm:mathertimedependent}.

\begin{defn}\label{def:minimizableL}
 Given a measure $\mu\in\mathscr H$, a \emph{$\mu$-minimizable Lagrangian} is a smooth function $\varphi\in C^\infty(T\R^n\times I)$ such that there is a Lipschitz function $f\colon \R^n\times I\to\R$ that is differentiable on $\pi(\supp\mu)$, and at all points $(x,t)\in\R^n\times I$  where $f$ is differentiable, satisfies
 \[\varphi(x,v,t)\geq df_{(x,t)}(v,\mathbf 1),\]
 with equality when $(x,v,t)\in\supp\mu$, and
 $\int df\,d\mu=0$.
\end{defn}

\subsection{Local characterization}\label{sec:characterizationvariations}

\begin{thm}\label{thm:variations}

 Let $\mu\in \mathscr H$ be measure closed in the open set $U\subset \R^n$, and let $\eta\in \mathscr E'$ be a compactly-supported distribution on $TU\times I$. Then $\mu$ can be deformed in the direction $\eta$
 if, and only if, $\eta$ satisfies 
 \[\langle\eta,\varphi\rangle\geq 0\]
 for all $\mu$-minimizable Lagrangians $\varphi$.
\end{thm}
\begin{proof}
 First we see that if $\eta$ is a distribution of order $m$ and if we let $W\subset T\R^n\times I$ be an open and bounded neighborhood of the compact set $\supp\mu$, then we can apply Lemma \ref{lem:boundednormability} with $\gamma=2|\eta|'_{\overline W,m}$. The lemma tells us that if we take $X$ to be the subset of $\mathscr E'$ consisting of compactly-supported distributions $\xi$ of order $\leq m$ and with norm $|\xi|'_{\overline W,m}< \gamma$, then the seminorm $|\cdot|'_{\overline W,m}$ is nondegenerate on $X$ and determines the weak* topology there. 
 
 Looking to apply Theorem \ref{thm:differentiablecurves}, set $V=\mathscr E'$, $C=\mathscr H$, $p=\mu$, and $v=\eta$, and from the theorem we will get $c(t)=\mu_t$. Note that the set $X$ contains all Radon measures $\nu$ with $|\nu|'_{\overline W,m}< \gamma$. It thus contains the intersection of the tangent cone of $C=\mathscr H$ at $p=\mu$ with the open subset of $\mathscr E'$ consisting of all distributions $\xi$ with $|\xi|'_{\overline W,m}< \gamma$, an open subset that also contains $v=\eta$. 
 
 Thus Theorem \ref{thm:differentiablecurves} can be applied, and in this case it tells us that a family $(\mu_t)_{t\geq 0}\subseteq \mathscr H$ with $\mu_0=\mu$ and $d\mu_t/dt|_{t=0}=\eta$ exists if, and only if, $\langle \eta, f\rangle\geq 0$ for all $f\in C^\infty(T\R^n\times I)=(\mathscr E')*$ such that $\langle \mu,f\rangle\leq \langle\nu,f\rangle$ for all $\nu\in \mathscr H$. By Theorem \ref{thm:mathertimedependent}, these functions $f$ are exactly the $\mu$-minimizable Lagrangians.
\end{proof}

\subsection{Examples}\label{sec:examplesvariations}

\begin{defn}\label{def:critical} 
 For $L\in C^k(\R^n\times I)$ is any smooth Lagrangian, its \emph{action} is the functional $A_L\colon\mathscr H\to\R$ given by $\nu\mapsto \int L\,d\nu$. The action $A_L$ is differentiable in the sense that for every family $(\mu_t)_{t\geq 0}\subset \mathscr H$ such that $d\mu_t/dt|_{t=0^+}$ exists and is equal to a distribution $\eta\in \mathscr E'$ of order $k$, we have 
 \[\left.\frac{d}{dt}A_L(\mu_t)\right|_{t=0^+}=\langle \eta,L\rangle.\] 
 
 If $\Theta\subseteq \mathscr E'$ is a set of distributions and $\mu$ is a measure in $\mathscr H$, we say that \emph{$\mu$ is $\Theta$-critical for the action of $L$} if, for all $\eta\in\Theta$ such that $\mu$ can be deformed in the direction $\eta$, we have $\langle \eta,L\rangle=0$.
\end{defn}

\begin{example}\label{ex:horizontal}
 A first example is given by the \emph{(velocity-independent) horizontal variations}, which are constructed as follows. Take a smooth vector field $P\colon \R^n\times I\to T(\R^n\times I)$ vanishing at times $0$ and $t_0$. Let $\Psi_s\colon \R^n\times I \to \R^n\times I$ be its flow. Assume that the differential of the flow, $D\Psi_s\colon T(\R^n\times I)\to T(\R^n\times I)$, preserves the set $T\R^n\times I\times \{\mathbf 1\}$. Let
 \begin{align*}
  &p\colon T\R^n\times I\times \{\mathbf 1\}\to T\R^n\times I\\
  &p(x,v,t,\mathbf 1)=(x,v,t).
 \end{align*}
 Then, given a measure $\mu\in \mathscr H$, the measure $\mu_s=p_*(D\Psi_s)^*p^*\mu$ is also in $\mathscr H$. The derivative at $s=0$ of the family $(\mu_s)_{s\in \R}$ is $\eta=(p_*DP)\mu$ (here, $p_*DP$ is a vector field on $TR^n\times I$, and $(p_*DP)\mu$ is the derivative of $\mu$ in the direction of $DP$, in the sense of distributions). Thus we know that $\mu$ can be deformed in the direction $\eta$. 
 
 Since, for any $\mu$-minimizable Lagrangian $\varphi$, the map $s\mapsto \langle\mu_s,\varphi\rangle$ is smooth, and reaches its minimum at $s=0$, the condition in Theorem \ref{thm:variations} is verified. 
 
 Horizontal variations will be further investigated in Section \ref{sec:horizontalvariations}.
\end{example}

\begin{example}\label{ex:vectorfieldvariations}
Our main example is a broad generalization of Example \ref{ex:horizontal}:
\begin{coro}\label{coro:vectorfieldvariations}
 Let $\mu\in\mathscr H$ be a measure closed in the open set $U\subset \R$, and let $X\colon T\R^n\times I\to T\R^n\times I$ be a Borel-measureable vector field supported in $TU\times I$.  
 Let $\eta_X$ be the distribution given by
 \[\langle\eta_X,\varphi\rangle=\int_{T\R^n\times I}X\varphi\,d\mu,\quad\varphi\in C^\infty(T\R^n\times I).\]
 Assume that, for all $\phi\in C^\infty(\R^n\times I)$, and with $\sigma(x,v,t)=((x,t)(v,\mathbf 1))$,
 \begin{equation}\label{eq:closednesspreserved}
  \langle\eta_X,d\phi\rangle=\int_{T\R^n\times I} X(d\phi\circ \sigma )\,d\mu=0.
 \end{equation}
 Then  $\mu$ can be deformed in the direction $\eta_X$.
\end{coro}
This Corollary tells us that $\mu$ can be flowed in the direction of $X$ if, and only if, $X$ preserves the closedness condition \eqref{eq:closedness}, which is the content of condition \eqref{eq:closednesspreserved}. 
Some consequences of Corollary \ref{coro:vectorfieldvariations} will be explored in Section \ref{sec:invariant} and Example \ref{ex:verticalvariations}.
\begin{proof}[Proof of Corollary \ref{coro:vectorfieldvariations}]
 This is a corollary to Theorem \ref{thm:variations}. The necessity of condition \ref{eq:closednesspreserved} is clear; let us prove its sufficiency. With the notations of Definition \ref{def:minimizableL}, for every $\mu$-minimizable Lagrangian $\varphi$, we have 
 \[
  \langle\eta_X,\varphi\rangle=\langle\eta_X,\varphi-df\rangle+\langle\eta_X,df\rangle=\langle\eta_X,\varphi-df\rangle,
 \]
 which vanishes since $\varphi-df\geq 0$ reaches its minimum everywhere on the support of $\mu$, so its derivative (which exists $\mu$-almost everywhere by the Lemma \ref{lem:thatextraregularity} below with $Y=(X,0)$) must be 0. %
\end{proof}

\begin{lem}\label{lem:thatextraregularity}
 Let $\mu\in \mathscr H$ be a measure on $T\R^n\times I$ closed in $U\subset \R^n$, and let $\varphi$ be a $\mu$-minimizable Lagrangian. Then, at $\mu$-almost every point $(x,v,t)\in T\R^n\times I$ and for every vector $Y\in T_{(x,v,t,\mathbf 1)}(T(\R^n\times I))$, the limit
 \[\lim_{\varepsilon\to 0}\frac{df((x,t, v,\mathbf 1)+\varepsilon Y)-df(x,t,v,\mathbf 1)}{\varepsilon}\]
 exists and is finite.
\end{lem}
\begin{proof}
 Let $\chi\colon T\R^n\times I\to\{0,1\}$ be a measurable function vanishing on a neighborhood of $T\R^n\times \{0,t_0\}$. The boundary $\partial(\chi\mu)$ of the current induced by the measure $\chi\mu$ is itself a signed measure $\nu$ on $\R^n\times I$, that is,  $\nu=\partial(\chi\mu)$ is a signed measure (supported away from $\supp\partial\mu$) that satisfies, for all $g\in C^\infty(\R^n\times I)$, 
 \[\int_{T\R^n\times I}\chi(x,v,t) dg_{(x,t)}(v,\mathbf 1)\,d\mu(x,v,t)=\int_{\R^n\times I} g(x,t)\,d\nu(x,t). \]
 To see why, note that this is true for measures induced by curves, and by the results of \cite{federerrealflat} we know that all measures in $\mathscr H$ can be approximated by those.
 
 We know that $f$ is differentiable on $\pi(\supp\mu)\setminus (T\R^n\times \{0,t_0\})$ from Definition \ref{def:minimizableL} (which is motivated by item \ref{it:fisdifferentiable} of Theorem \ref{thm:mathertimedependent}). This in turn implies the existence of $\lim_{\varepsilon\searrow 0}\frac1\varepsilon(f\circ \Phi_\varepsilon-f)$ on $\supp\nu\subset\pi(\supp\mu)\setminus (T\R^n\times \{0,t_0\})$
 for every smooth family of diffeomorphisms $\Phi_s\colon \R^n\times I\to\R^n\times I$, $0\leq s\leq 1$, with $\Phi_0=\id$. Thus we have 
 \begin{multline*}
  \lim_{\varepsilon\searrow 0}\frac1\varepsilon\langle\chi\mu,d(f\circ \Phi_\varepsilon)-df\rangle=\lim_{\varepsilon\searrow 0}\frac1\varepsilon\langle\partial(\chi\mu),f\circ \Phi_\varepsilon-f\rangle=\\
  \left\langle\nu,\lim_{\varepsilon\searrow 0}\frac1\varepsilon(f\circ \Phi_\varepsilon-f)\right\rangle=\left\langle \nu,df\circ \left.\frac{d\Phi_\varepsilon}{d\varepsilon}\right|_{\varepsilon=0}\right\rangle.
 \end{multline*}
 
 This, together with the fact that $df$ is linear on each fiber of $T\R^n\times I$ implies the statement of the lemma.
\end{proof} 
\end{example}

\begin{example}\label{ex:verticalvariations}
 Let $\mu\in\mathscr H$.
 Let $\mathcal H=L^2(T\R^n\times I,\R^n;\mu)$ be the Hilbert space of measurable vector fields $X\colon T\R^n\times I\to \R^n$ %
 with $\int |X|^2 d\mu<+\infty$. 
 We extend $X\in \mathcal H$ to a vectorfield $\tilde X$ on $T(\R^n\times I)$ by letting $\tilde X(x,v,t,\tau)=(0,X(x,v,t),0,0)$.
 Thus, when acting on a function $\phi\in C^\infty(T\R^n\times I)$, $\tilde X\phi=\phi_v\cdot X$.
 For $X\in \mathcal H$, let $\eta_X$ be the distribution given by 
 \[\langle\eta_X,\varphi\rangle=\int \tilde X\varphi\,d\mu=\int d\varphi(\tilde X)\,d\mu,\quad \varphi\in C^\infty(T\R^n\times I).\]
 It follows from Corollary \ref{coro:vectorfieldvariations} that $\mu$ can be deformed in the direction $\eta_X$ if, and only if, for all $\phi\in C^\infty(\R^n\times I)$ we have
 \begin{equation}\label{eq:holonomicitysmall}
  \langle\eta_X,d\phi\rangle=0.
 \end{equation} 

 Let us consider $\Theta_1$-criticality with $\Theta_1=\{\eta_X:X\in\mathcal H\}$. We introduce a piece of notation: for $u\in C^\infty(\R^n\times I)$, we let $\nabla u$ denote the vector field that satisfies $\nabla u\cdot X=Xu$ for all $X$, so that $\nabla u$ is the gradient in the $v$ directions; in coordinates, $\nabla u=(\partial u/\partial v_1,\dots, \partial u/\partial v_n)$. 
 With Definition \ref{def:critical}, for $L\in C^\infty(\R^n\times I)$ and $\mu\in \mathscr H$, we see that $\mu$ is $\Theta_1$-critical for the action of $L$ if $0=\langle \eta_X,L\rangle=\int \tilde XL\,d\mu=\int \nabla L\cdot X\,d\mu$ for all $X\in \mathscr H$ that satisfies $0=\langle\eta_X,d\phi\rangle=\int \nabla d\phi\cdot X\,d\mu$ for all $C^\infty(\R^n\times I)$. Note that both $\nabla L$ and $\nabla d\phi$ are elements of $\mathcal H$. Thus, the vector field $X$ must be contained in the perpendicular space to the linear space of gradients of exact forms $\nabla d\phi$, within $\mathcal H$, while $\nabla L$ must be contained in the perpendicular space to $X$; in other words, $\nabla L\in (\nabla d C^\infty(\R^n\times I))^{\perp\perp}=\overline{\nabla d C^\infty(\R^n\times I)}$. This means that \emph{the momenta $\partial L/\partial v$ coincide with an exact form (with only $L^2$-regularity) if, and only if, $\mu$ is $\Theta_1$-critical for the action of $L$}.
\end{example}

\begin{example}\label{ex:transpositional}
 The following example works only in the time-independent setting, so we will drop the $I$ factor altogether; the theory can be adapted easily. See also \cite{myminimizable} for more details on this context.

 Let $\mu$ be a measure in $\mathscr H$, and let $\eta_p\in\mathscr E'$ be the distribution given by 
 \[\eta_p=\delta_p-v\cdot\partial_v\delta_p-\frac{1}{\mu(T\R^n)}\mu\]
 for some point $p=(x,v)\in \supp\mu\subset T\R^n$, $\delta_p$ the Dirac delta at $p$, $v\cdot \partial_v\delta_p$ the directional derivative in the fiberwise direction $v$ at $p$, and $\partial_t\delta_p$ the derivative in the time direction at $p$. %
 In other words, for $u\in C^\infty(T\R^n)$, $\eta_p$ acts as
 \[\langle\eta_p,u\rangle=u(p)-v\cdot\frac{\partial u}{\partial v}(p)-\frac{1}{\mu(T\R^n)}\int_{T\R^n} u\,d\mu.\]
 If $\varphi$ is a $\mu$-minimizable Lagrangian  (but is time-independent), we have
 \[\langle\eta_p,\varphi\rangle=\varphi(p)-v\cdot\frac{\partial\varphi}{\partial v}(p)-\frac{1}{\mu(T\R^n)}\int_{T\R^n}\varphi\,d\mu=0,\]
 because $\varphi$ and $v\cdot\partial \varphi/\partial v$ both coincide on $\supp\mu$ and $\int\varphi\,d\mu=0$,
 so the condition in Theorem \ref{thm:variations} is satisfied, and $\mu$ can be deformed in the direction $\eta_p$.
 
 Let us see what $\Theta_2$-criticality means for $\Theta_2=\{\eta_p:p\in T\R^n\}$. If $L\in C^\infty(T\R^n)$, $\mu$ is critical for the action of $L$ if, and only if, for all $p=(x,v,t)\in\supp\mu$, we have
 \begin{equation}\label{eq:energyconservation}
  0=\langle\eta_p,L\rangle=L(x,v,t)-v\frac{\partial L}{\partial v}(x,v,t)-\frac{1}{\mu(T\R^n)}\int L\,d\mu.
 \end{equation}
 Here, we may identify $\partial L/\partial v$ as the momentum, and the expression
 \[v\frac{\partial L}{\partial v}-L\]
 as the Hamiltonian, so that $\Theta_2$-criticality is equivalent to a sort of energy conservation law \eqref{eq:energyconservation}.
 
 Recall that energy conservation was already known for minimizers \cite[Corollary 4]{myminimizable} and for invariants under the flow of the Hamiltonian vector field. The difference here is that measures that satisfy $\Theta_2$-criticality are not necessarily minimizers, and the Hamiltonian may not be defined in this context (because we do not ask for $L$ to be convex).
\end{example}

\begin{example}\label{ex:mather}
 For a point $p=(x,v,t)\in T\R^n\times I$ and a vector $w\in T_x\R^n$, let $\eta_{p,w}\in\mathscr E'$ be the distribution given by
 \[\eta_{p,w}=-\partial_w^2\delta_p,\]
 so that for $u\in C^\infty(T\R^n\times I)$ the distribution $\eta$ acts as
 \[\langle\eta,u\rangle=-\frac{d^2u}{\partial v^2}(p)(w,w)=-\frac{d^2u}{\partial w^2}(p).\]
 In other words, $\langle\eta,u\rangle$ is minus the Hessian of the restriction of $u$ to $T_{x}\R^n\times \{t\}$ evaluated (as a quadratic form) on $w$. 
 
 We consider the question of the meaning of $\Theta_3$ criticality for $\Theta_3=\{\eta_{p,w}:p\in T\R^n\times I,\, w\in T_{\pi(p)}\R^n\}$. This is especially interesting when we know that $L$ has positive-definite %
 fibrewise Hessian in a certain region, because in that case $\langle \eta_{p,w},L\rangle< 0$, so in order for $\mu$ to be $\Theta_3$-critical for the action of such $L$, it would be necessary that $\mu$ \emph{cannot} be deformed in the direction $\eta_{p,w}$.
 In order for the latter to happen, by Theorem \ref{thm:variations} it follows that it is necessary that there exist a $\mu$-minimizable Lagrangian $\varphi$ that obstructs the deformability, in the sense that $\langle\eta_{p,w},\varphi\rangle<  0$.
 This would mean that, if $f$ is as in Definition \ref{def:minimizableL}, we would have
 \[0> \langle\eta_{p,w},\varphi\rangle=\langle\eta_{p,w},\varphi-df\rangle=-\frac{\partial^2}{\partial w^2}(\varphi-df)(p).\]
 Note that $\supp\mu$ is contained in the zeroes of $\varphi-df$, while those zeroes are contained in the zeroes of $\partial(\varphi-df)/\partial w$ since $\varphi\geq df$.
 By item \ref{it:lipschitzity} of Theorem \ref{thm:mathertimedependent}, $df|_{\supp\mu}$ can be extended to a locally Lipschitz form $\varsigma$ on $\R^n\times (0,t_0)$ (which may, however, not be exact), and after taking such an extension, $\supp\mu$ continues to be contained in the zeroes of $\varphi-\varsigma$ and of $\partial(\varphi-\varsigma)/\partial w$. Since 
 \[\frac{\partial^2}{\partial w^2}(\varphi-\varsigma)(p)= 
  \frac{\partial^2}{\partial w^2}(\varphi-df)(p)\neq 0
 \]
 and since $\partial(\varphi-\varsigma)/\partial w$ is smooth on the fibers, the implication is, by \cite[Theorem 1.3]{wuertz}, that the support of $\mu$ must be contained in a locally Lipschitz hypersurface of $T\R^n\times (0,t_0)$. Applying this reasoning to several, linearly-independent vectors $w$, we conclude that \emph{the support of $\mu$ must be contained in the graph of a section $\varsigma\colon \R^n\times I\to T\R^n\times I$ that is locally Lipschitz on $\R^n\times(0,t_0)$}. Otherwise $\mu$ would be deformable in the direction $\eta_{p,w}$ for some $p$ and $w$, and $\mu$ would not be $\Theta_3$-critical for the action of $L$.
 
 We observe that the same applies for negative-definite Lagrangians $L$, replacing $\eta_{p,w}$ with $-\eta_{p,w}$.
 
 We finish with the remark that deforming in the direction $\eta_{p,w}$ is a sort of infinitesimal analogue of what Mather did in his paper \cite[Lemma in \S4]{matheractionminimizing91}. In that paper, he replaces a pair of curves $\alpha$ and $\beta$ with a new pair of curves $a$ and $b$, cleverly constructed so that the velocity of $a$ and $b$ approaches the average of the velocities of the original curves $\alpha$ and $\beta$. That is,
 \[a'\approx\tfrac12(\alpha'+\beta')\approx b'.\]
 Thus, in this case the difference of the action of the pairs $(\alpha,\beta)$ and $(a,b)$ is approximately
 \[L(\alpha')+L(\beta')-2L(\tfrac12(\alpha'+\beta')).\]
 Taking $w=\alpha'-\beta'$, this reminds us a lot of the classical approximation of the second derivative,
 \[\frac{\partial^2 L}{\partial w^2}(p)=\lim_{\delta\to 0}\frac{L(p+\delta w)+L(p-\delta w)-2L(p)}{\delta^2}.\]
 Whence it is not so surprising that the regularity that Mather found is precisely the consequence of $\Theta_3$-criticality.
\end{example}

\section{Invariant measures}
\label{sec:invariant}
\subsection{The Euler-Lagrange flow and the definition of invariance}
\label{sec:eulerlagrangeflow}

As in Section \ref{sec:characterizationvariations}, let $U$ be an open subset of $\R^n$, $I=[0,t_0]$ a time interval, and $\mathscr H$ the space of measures closed in $U$. We will generally denote points in $TU\times I$ as $(x,v,t)$, where $x\in \R^n$, $v\in T_xU\cong\R^n$, and $t\in I$. 

Let $L\in C^2(TU\times I)$ be an arbitrary function. We will assume that the derivatives of $L$ of degree 2 are locally Lipschitz. 
Recall that a $C^2$ curve $\gamma\colon I\to\R^n$ satisfies the Euler-Lagrange equations if
\begin{align*}
 L_x(\gamma(t),\gamma'(t),t)&=\frac{d}{dt}L_v(\gamma(t),\gamma'(t),t),\quad t\in (0,t_0).
 \end{align*}
 which we can rewrite as
 \begin{align*}
  L_x&=L_{vt}+L_{xv}\gamma'(t)+L_{vv}\gamma''(t).
\end{align*}
This equation can be solved for $\gamma''$ if $L_{vv}$ is invertible; this inspires the following definition.

\begin{defn}[Euler-Lagrange flow]\label{def:ELflow}
 Assume that $L_{vv}$ is an invertible matrix everywhere on an open set  $V\subseteq TU\times I$.
 The \emph{Euler-Lagrange vector field} on $V$ is given by
 \begin{align*}
  \mathcal X(x,v,t)&=\frac{\partial}{\partial t}+v\frac{\partial}{\partial x}+(L_x-L_{vt}-L_{xv}v)L_{vv}^{-1}\frac{\partial}{\partial v}\\
  &\notag=\frac{\partial}{\partial t}+\sum_iv_i\frac{\partial}{\partial x_i}+\sum_{i,k}\Big(L_{x_i}-L_{v_it}-\sum_jv_jL_{x_jv_i}\Big)(L_{vv}^{-1})_{ik}\frac\partial{\partial v_k}.
 \end{align*}
 Since the second derivatives of $L$ are locally Lipschitz and $L_{vv}$ is invertible everywhere on $V$, $\mathcal X$ is locally Lipschitz too. The flow $\Phi_s$ of $\mathcal X$ is thus locally well defined for small times $s\in \R$ (yet it may be an incomplete flow), and we will call it \emph{Euler-Lagrange flow} on $V$. It satisfies 
 \[\Phi_0=\mathrm{identity}\quad\textrm{and}\quad\frac{d}{dt}\Phi_s(x,v,t)=\mathcal X\circ \Phi_s(x,v,t)\]
 for all $(x,v,t)\in T\R^n\times I$ and all $s\in \R$ for which it is defined.
\end{defn}

\begin{defn}[Invariant]\label{def:invariant}
 A compactly-supported, positive, Radon measure $\mu\in \mathscr H$ on $TU\times I$ is \emph{invariant} under the Euler-Lagrange flow if for all compactly-supported $\varphi\in C_c^\infty(TU\times I)$ we have 
 \[\int\mathcal X\varphi\,d\mu=0.\]
\end{defn}

Note that if the flow $\Phi_s$ is complete, $\mu$ being invariant is equivalent to having, for all $s\in \R$, $\Phi_s^*\mu=\mu$. The equivalence follows from
\[\frac d{ds}\int \varphi\,d(\Phi_s^*\mu)=\frac d{ds}\int(\Phi_s)_*\varphi\,d\mu=\frac d{ds}\int\varphi\circ\Phi_s^{-1}\,d\mu=\int-\mathcal X\varphi \,d\mu.\]

\subsection{Impossibility examples}
\subsubsection{Impossibility of a purely variational characterization}
\label{sec:impossiblityexample}

In the following, we will seek a variational characterization of measures invariant under the Euler-Lagrange flow. That is to say, we would like to find a class $\Omega$ of distributions such that if $\mu$ is $\Omega$-critical for $L$, then $\mu$ is invariant under the Euler-Lagrange flow. The following example shows that this cannot be achieved in general.

\begin{example}\label{ex:noninvariantminimum}
 Let $U=(0,1)\times(0,1)$.
 Let $h\colon\R^2\to\R$ be a nonnegative function vanishing only at the points $(1,1)$ and $(1,-1)$, where we assume that $\det L_{vv}>0$. For $(x,v,t)\in T\R^2\times I$, $I=[0,1]$, let $L(x,v,t)=h(v)$, $V=\{(x,v,t)\in T\R^2\times I:\det h_{vv}\neq0\}$, and let $\mu$ be the measure 
 \[\mu=\tfrac12(\delta_{((t,0),(1,1),t)}+\delta_{((t,0),(1,-1),t)})\otimes \lebesgue_{[0,1]},\] 
 where $\lebesgue_{[0,1]}$ stands for the Lebesgue measure on $I$. In other words, for any measurable function $\varphi$ on $T\R^2\times I$ we have 
 \[\int_{T\R^2\times I}\varphi\,d\mu=\tfrac12\int_0^1\varphi((t,0),(1,1),t)+\varphi((t,0),(1,-1),t)\,dt.\]
 In particular, $\int_{T\R^2\times I}L\,d\mu=\frac12\int_0^1h(1,1)+h(1,-1)\,dt=0$. Thus, within $\mathscr H$, $\mu$ is a minimizer of the action of $L$.
 
 The importance of this example is that \emph{$\mu$ is not invariant under the Euler-Lagrange flow, although it satisfies
 $\frac d{dt} \int L\,d\mu_t|_{t=0+}\geq0$ for every variation $(\mu_t)_{t\geq 0}$ of $\mu$}. An immediate consequence is that it would be impossible to characterize invariant measures by purely variational means, since there are non-invariant measures that are critical with respect to all possible variations (in the example this criticality happens because the measure is globally minimal).
 
 To see that $\mu$ is not invariant under the Euler-Lagrange flow, we write down the corresponding vector field explicitly for $(x,v,t)\in\supp\mu$:
 \[\mathcal X(x,v,t)=\frac{\partial}{\partial t}+v_1\frac{\partial}{\partial x_1}+v_2\frac{\partial}{\partial x_2}.\]
 This follows from the fact that $L=h(v)$ satisfies $L_x=L_{vt}=L_{xv}=0$.
 Take $\varphi\in C^\infty_c(TU\times I)$ vanishing at $t=0,1$; then 
 \begin{align*}
  \int \mathcal X\varphi\,d\mu&=\tfrac12\int_0^1\varphi_t((t,0),(1,1),t)+\varphi_t((t,0),(1,-1),t)\,dt\\
  &\quad+\tfrac12\int_0^1\varphi_{x_1}((t,0),(1,1),t)+\varphi_{x_1}((t,0),(1,-1),t)\,dt\\
  &\quad+\tfrac12\int_0^1\varphi_{x_2}((t,0),(1,1),t)-\varphi_{x_2}((t,0),(1,-1),t)\,dt.
 \end{align*}
 While the first term (involving the integral of $\varphi_t$) vanishes by the fundamental theorem of calculus and the fact that $\varphi$ vanishes at the endpoints, the other two terms do not vanish for general $\varphi$. So indeed $\mu$ is not invariant under the Euler-Lagrange flow.
\end{example}
\begin{rmk}
 If the Lagrangian $L$ is strictly convex in the fibers, all minimal measures in $\mathscr H$ are invariant under the Euler-Lagrange flow; see \cite{contrerasiturriagabook} or Section \ref{sec:manethm}.
\end{rmk}

\subsection{Closed velocities}
\label{sec:closedvelocities}

 Although, as shown in the previous section, a purely variational characteriztion of invariance under the Euler-Lagrange flow does not exist, it does exist if we restrict to an appropriate class of measures. For this purpose we have found a very large class that includes all (compactly-supported) measures that are invariant under the Euler-Lagrange flow, for any Lagrangian $L$; the class is defined by either of the equivalent properties of the following proposition --- see Definition \ref{def:closedderivatives} below.
 
 \begin{prop}\label{prop:closedvelocities}
  Let $U$ be an open subset of $\R^n$, $I=[0,t_0]$ a time interval, and $\mathscr H$ the space of measures closed in $U$. Let $(x,v,t)$ be coordinates on $T\R^n\times I$, with $x\in\R^n$, $v\in T_x\R^n$, $t\in I$.%
 
 Let $\mu\in \mathscr H$ be a closed measure. 
 Then either all or none of the following are true:
 \begin{enumerate}
  \item\label{it:closedvelocityexists} There exists a function $C\colon TU\times I\to \R^n$ such that, for functions $f\in C^\infty(TU\times I)$ that vanish 
  on $t=0$ and $t=t_0$,
  we have
  \[
   \int df_{(x,v,t)}\left(v\frac{\partial}{\partial x}+C(x,v,t)\frac{\partial}{\partial v}+\frac\partial{\partial t}\right)\,d\mu=0.
  \]
  
  \item\label{it:closedvelocitylifting} There exists a compactly-supported measure $\tilde \mu$ on $T(T\R^n\times I)$ such that
   \begin{enumerate}
    \item\label{it:higherclosed} $\tilde\mu$ is closed on $TU\times I$, that is, such that
    \[\int_{\R^n\times I} df\,d\tilde\mu=0\]
    for all $f\in C^\infty(T\R^n\times I)$ vanishing on the boundary $T\R^n\times \{0,t_0\}$,
    \item\label{it:holonomicset} $\tilde\mu$ is supported on the set of points $(x,v,t,\tilde x,\tilde v,\tilde t)\in T(T\R^n\times I)$ (where we use the notations $v,\tilde x\in T_x\R^n\cong\R^n$, $\tilde v\in T_vT_x\R^n\cong\R^n$, $\tilde t\in T_tI\cong \R$) such that 
    \[\tilde x=v,\]
    \item\label{it:lifting}the pushforward of $\tilde\mu$ by the projection $\pi_{T\R^n\times I}\colon T(T\R^n\times I)\to T\R^n\times I$ equals $\mu$, that is,
    \[(\pi_{T\R^n\times I})_*\tilde\mu=\mu.\]
   \end{enumerate}
  \end{enumerate}
 \end{prop}
 \begin{defn}\label{def:closedderivatives}\label{def:secondderivatives}
  A closed measure $\mu\in \mathscr H$ that satisfies either (and hence both) of the conditions in the statement of Proposition \ref{prop:closedvelocities} will be called \emph{a closed measure with closed velocity}. We will refer to the function $C$ in item \ref{it:closedvelocityexists} as a \emph{second derivative} of $\mu$, a piece of terminology motivated by Example \ref{ex:verynicecurve}. We will denote by $\mathscr H_\Delta$ the subset of $\mathscr H$ consisting of all closed measures with closed velocity.
 \end{defn}
 
\begin{rmk}
 \emph{The set $\mathscr H_\Delta$ of closed measures with closed velocity is not closed in the topology of $\mathscr E'$, but it is dense in the set of closed measures $\mathscr H\subset \mathscr E'$, $\overline{\mathscr H_\Delta}=\mathscr H$.} To see that $\mathscr H$ is dense, note that the set $\mathscr H$ is the closed convex hull of the measures induced by curves as in \eqref{eq:inducedclosedmeasure}, which are all of closed velocity, as is shown in Example \ref{ex:verynicecurve}; see Remark \ref{rmk:approximability}. (A similar argument follows from the observation that the set of closed measures consisting of smooth densities on $T\R^n\times I$ is dense in $\mathscr H$; compare with Example \ref{ex:smoothisclosedvelocity}.) This, together with the existence of measures in $\mathscr H$ that are not of closed velocity, as  those of Examples \ref{ex:notsonicecurve} and \ref{ex:notclosedvelocity}, implies that the set of closed measures with closed velocity is not closed. 
\end{rmk}
\begin{rmk}
 \emph{The set $\mathscr H_\Delta$ is convex.} Indeed, if $\mu$ and $\nu$ are closed measures with closed velocities with $\tilde\mu$ and $\tilde \nu$ as in item \ref{it:closedvelocitylifting} in Proposition \ref{prop:closedvelocities}, then $(1-\lambda)\mu+\lambda\nu$, $\lambda\in(0,1)$, is also a closed measure with closed velocity, as the measure $(1-\lambda)\tilde\mu+\lambda\tilde \nu$ again works as in item \ref{it:closedvelocitylifting}.
\end{rmk}
 \begin{proof}[Proof of Proposition \ref{prop:closedvelocities}]
  If item \ref{it:closedvelocityexists} holds, then let $y\colon TU\times I\to T(T\R^n\times I)$ be given by 
  \[y(x,v,t)=v\frac{\partial}{\partial x}+C(x,v,t)\frac{\partial}{\partial v}+\frac{\partial}{\partial t}=(x,v,t,v,C(x,v,t),\mathbf 1).\]
  Letting $\tilde \mu=y_*\mu$, we get a measure $\tilde\mu$ as in item \ref{it:closedvelocitylifting}.
  
  Now assume instead that item \ref{it:closedvelocitylifting} holds. Disintegrate $\tilde\mu$ along the fibres of the projection $\pi_{T\R^n\times I}\colon T(T\R^n\times I)\to T\R^n\times I$, like so:
  \[\tilde\mu=\int_{T\R^n\times I}\nu_pd\eta(p)\]
  for some measure $\eta$ on $T\R^n\times I$ and some collection $\{\nu_p\}_{p\in T\R^n\times I}$ of measures $\nu_p$ on the tangent spaces $\pi^{-1}_{T\R^n\times I}(p)\cong T_{p}(T\R^n\times I)\cong \R^{2n+1}$; this in turn means that for measurable $f\colon T(T\R^n\times I)\to\R$ we have
  \[\int_{T(T\R^n\times I)}f\,d\tilde\mu=\int_{T\R^n\times I}\left[\int_{\pi_{T\R^n\times I}^{-1}(p)}f\,d\nu_p\right]d\eta(p).\] 
  Let, for $p=(x,v,t)\in T\R^n\times I$, 
  \[C(p)=\int_{\pi^{-1}_{T\R^n\times I}(p)} \tilde v\,d\nu_p(\tilde x, \tilde t,\tilde v)\]
  be the center of mass of $\nu_p$ on the corresponding fiber.
  It follows from item \ref{it:holonomicset} that $\int\tilde x\,d\nu_p=v$, from item \ref{it:lifting} that $\int \tilde t\,d\nu_p=0$, and from item \ref{it:higherclosed} that $C$ satisfies the condition from item \ref{it:closedvelocityexists}.
 \end{proof}
 
\begin{example}\label{ex:verynicecurve}
 Let $\gamma\colon [0,t_0]=I\to\R^n$ be a $C^{1,1}$ curve; in other words, $\gamma$ is differentiable and its derivative $\gamma'$ is Lipschitz. Then if we let $\mu$ be the measure on $T\R^n\times I$ defined by 
 \begin{equation}\label{eq:inducedclosedmeasure}
  \int f\,d\mu=\int_0^{t_0}f(\gamma(t),\gamma'(t),t)\,dt,\qquad f\in C^\infty(T\R^n\times I),
 \end{equation}
 the measure $\mu$ is closed with closed velocity, i.e., $\mu\in\mathscr H_\Delta$. The function $C$ and the measure $\tilde\mu$ in Proposition \ref{prop:closedvelocities} are in this case given by
 \begin{equation}\label{eq:inducedtildemeasure}
  C(\gamma(t),\gamma'(t),t)=\gamma''(t),\qquad t\in [0,t_0],
 \end{equation}
 and $C\equiv0$ elsewhere; and
 \[\int_{T(T\R^n\times I)}f\,d\tilde\mu=\int_0^{t_0} f(\gamma(t),\gamma'(t),t,\gamma'(t),\gamma''(t),\mathbf 1)\,dt,\]
 where $f\in C^\infty(T(T\R^n\times I))$, $\gamma''(t)\in T_{\gamma'(t)}(T_\gamma(t)\R^n)\cong\R^n$, $\mathbf 1\in T_tI$. Recall that the second  derivative $\gamma''$ is defined for almost every $t$ because $\gamma'$ is Lipschitz, so we can apply Rademacher's theorem.

\end{example}
\begin{example}\label{ex:notsonicecurve}
 If now we take a curve $\gamma\colon[0,t_0]\to\R^n$ whose derivative is not continuous at one point, then the construction in Example \ref{ex:verynicecurve} produces a closed measure that is not of closed velocity. To see why, note that in this case the discontinuity in the image of the curve $(\gamma,\gamma')$ in $T\R\times I$ means that item \ref{it:closedvelocitylifting} cannot be satisfied.
\end{example}
\begin{example}\label{ex:notclosedvelocity}
 The measure $\mu$ in Example \ref{ex:noninvariantminimum} can be seen not to be of closed velocity because no lifting $\tilde\mu$ of $\mu$ will satisfy item \ref{it:holonomicset} in Proposition \ref{prop:closedvelocities}.
\end{example}
\begin{example}\label{ex:invariantisclosedvelocity}
 \emph{Every closed measure $\mu\in\mathscr H$ invariant under the Euler-Lagrange flow is of closed velocity and the vector field $(L_x-L_{vt}-vL_{xv})L^{-1}_{vv}$  is a second derivative for $\mu$.} 
This follows immediately from Definitions \ref{def:ELflow}, \ref{def:invariant}, and \ref{def:closedderivatives}.

\end{example}
\begin{example}\label{ex:smoothisclosedvelocity}
 \emph{A closed measure $\mu\in \mathscr H$ that consists of a smooth density $\mu=\rho \, d\mathscr L_{\R^{2n+1}}$ on $T\R^n\times I$ is of closed velocity.} Here $d\mathscr L_{\R^{2n+1}}$ denotes Lebesgue measure. Indeed, in this case there exists a smooth vector field on $T\R^n\times I$ of the form
 \[
 X(x,v,t)=\frac{\partial}{\partial t}+v\frac{\partial}{\partial x}+g(x,v,t)\frac{\partial}{\partial v},
 \]
 whose flow within $T\R^n\times (0,t_0)$ preserves $\mu$, which means that the equation satisfied by $X$ is $\divergence(\mu X)=0$. Then $\mu$ can be decomposed as a convex combination of measures induced as in \eqref{eq:inducedclosedmeasure} by orbits of the vector field $X$, and $g$ is a second derivative for $\mu$. 
\end{example}

\subsection{The class of exact horizontal variations} 
\label{sec:horizontalvariations}

We want to define the class of distributions $\Theta_\Delta(L)$ whose criticality characterizes Euler-Lagrange invariance. This first requires us to look at a larger class of horizontal variations.

 \subsubsection{Velocity-dependent horizontal variations}
 \label{sec:velocitydependenthv}

We generalize Example \ref{ex:horizontal}.
Let us consider a smooth vector field $F\colon T\R^n\times I\to\R^n$ vanishing on the boundary $T\R^n\times \{0,t_0\}$. We %
take a smoothly-varying collection $\{\gamma_s\}_{s\in(-\varepsilon,\varepsilon)}$ of smooth curves $\gamma_s\colon I\to\R^n\times I$ such that we have
\[\frac{d\gamma_s(t)}{ds}=F(\gamma_s(t),\gamma'_s(t),t).\]
Then also
\[\frac{d\gamma_s'(t)}{ds}=\frac{d}{dt}F(\gamma_s(t),\gamma'_s(t),t)=F_x\gamma'_s(t)+F_v\gamma''_s(t)+F_t.\]
Thus the family of curves $\{t\mapsto(\gamma_s(t),\gamma'_s(t))\}_s$ on $T\R^n$ is flowing along the vector field $(F,F_xv+F_vC+F_t)$ with $C=\gamma_s''(t)$.

In order to appropriately generalize this, let us take a closed measure $\mu\in\mathscr H_\Delta$ on $T\R^n\times I$ with closed velocity, and let us see how we can make it flow along the flow induced by $F$. Since $\mu$ has closed velocity, a second derivative $C\colon T\R^n\times I\to\R^n$ exists for $\mu$; see Definition \ref{def:closedderivatives}. 
We let $\eta_{F,\mu,C}\in \mathscr E'$ be the compactly-supported distribution on $T\R^n\times I$ given, 
for $f\in C^\infty(T\R^n\times I)$, by
\begin{equation}\label{eq:etaFmuC}
 \langle \eta_{F,\mu,C},f\rangle=\int_{T\R^n\times I} \left[ F\cdot \frac{\partial f}{\partial x}+(F_xv+F_vC+F_t)\cdot\frac{\partial f}{\partial v}%
 \right]d\mu.
\end{equation}

The distribution $\eta_{F,\mu,C}$ corresponds to flowing $\mu$ in the direction induced by $F$. We will need to assume that the flow of $F$ preserves the closedness of $\mu$, which in terms of $\eta$ means that, for all $\varphi\in C^\infty(\R^n\times I)$,
\begin{equation}\label{eq:etaclosed}
 \langle\eta_{F,\mu,C},d\varphi\rangle=0.
\end{equation}

Let us now use Corollary \ref{coro:vectorfieldvariations} to show that a family $(\mu_s)_{s\in\R}\subset\mathscr H$ exists with derivative $d\mu_s/ds|_{s=0}=\eta_{F,\mu,C}$. We need to show that for all $\phi\in C^\infty(\R^n\times I)$ we have $\langle\eta_{F,\mu,C},d\phi\rangle=0$. Indeed, 
\begin{align*}
 \langle\eta_{F,\mu,C},d\phi\rangle&=\int F\frac{\partial d\phi%
 }{\partial x}+d\phi_{(x,t)}(F_xv+F_vC+F_t,0)%
 \,d\mu\\
 &=\int \left[F\phi_{xx}v+F\phi_{xt}\mathbf 1+\phi_xF_xv+\phi_xF_vC+\phi_xF_t+\phi_t\cdot 0\right]d\mu\\
 &=\int Y(d\phi(F,0))\,d\mu
\end{align*}
where 
\[Y=Y(x,v,t)=v\frac{\partial}{\partial x}+C(x,v,t)\frac{\partial}{\partial v}+\frac{\partial}{\partial t}.\]
The last integral vanishes because $C$ is a second derivative (see item \ref{it:closedvelocityexists} of Proposition \ref{prop:closedvelocities}).

\begin{defn}
 Let $\mu\in\mathscr H_\Delta$ be a closed measure with closed velocity.
 A variation $(\mu_t)_{t\in\R}$ of $\mu=\mu_0$ is a \emph{velocity-dependent horizontal variation} if there exist a vector field $F\colon T\R^n\times I\to\R^n$ and a second derivative $C$ of $\mu$ such that $d\mu_t/dt|_{t=0}=\eta_{F,\mu,C}$, where $\eta_{F,\mu,C}$ is the distribution defined in \eqref{eq:etaFmuC}. Let $\Theta_\Delta$ be the set of distributions $\eta_{F,\mu,C}$ that are derivatives of velocity dependent horizontal variations.
\end{defn}

\subsubsection{Horizontal criticality}
\label{sec:horizontalcriticality}

Let $\pi\colon T\R^n\times I\to \R^n\times I$ be the projection, $\pi_{\R^n\times I}(x,v,t)=(x,t)$.
For the Lagrangian function $L\in C^2(T\R^n\times I)$, we define a special set of velocity dependent horizontal variations: let $\Theta_\Delta(L)$ be the set of distributions of the form $\eta_{F,\mu,C}$, defined as in \eqref{eq:etaFmuC}, corresponding to closed measures $\mu\in\mathscr H_\Delta$ with closed velocity, second derivatives $C$, and vector fields $F\in C^1(T\R^n\times I)$ with the special property that the restriction of $L_{vv}F$ to each fiber $\pi_{\R^n\times I}^{-1}(x,t)\cong\R^n$ is a closed form, that is,
\[d_v(L_{vv}F)=0.\]

\begin{defn}\label{def:horizontalcriticality}
 A variation $(\mu_s)_{s\in\R}$ with derivative $d\mu_s/ds|_{s=0}$ contained in $\Theta_\Delta(L)$ is an \emph{exact horizontal variation}.
 
 Given a Lagrangian $L\in C^2(T\R^n\times I)$, a closed measure $\mu\in \mathscr H_\Delta$ with closed velocities is \emph{horizontally critical} if it is $\Theta_\Delta(L)$-critical (see Definition \ref{def:critical}).
\end{defn}

\begin{rmk}
 The set of exact horizontal variations contains all velocity-in\-de\-pen\-dent horizontal variations, defined in Example \ref{ex:horizontal}. For, the corresponding vector field $F=p_*DP$ is linear in the velocity $v$ on each fibre of the tangent bundle, so $L_{vv}F$ is also linear in $v$, and it corresponds to the derivative of a cuadratic function of $v$.
\end{rmk}

\subsection{Characterization}
\label{sec:invariancecharacterization}

\begin{thm}\label{thm:invariance}
 Let $\mu$ be a compactly-supported, positive, Radon measure on $T\R^n\times I$, and let $L\in C^2(T\R^n\times I)$. The following are equivalent:
 \begin{enumerate}
  \item \label{it:invariant} The measure $\mu$ is invariant under the Euler-Lagrange flow.
  \item\label{it:horizontallycritical} The measure $\mu$ is closed with closed velocities, and it is horizontally critical.
 \end{enumerate}
\end{thm}

\begin{proof}[Proof of Theorem \ref{thm:invariance}]
 First, we need a preliminary computation. Let $f$ be a function in $C^\infty_c(T\R^n\times I)$. Assume for now that $C$ is a second derivative for $\mu$, $X=(L_x-L_{vt}-vL_{xv})L_{vv}^{-1}$, $\mathcal X$ the Euler-Lagrange vector field as in Definition \ref{def:ELflow}, that is,
 \begin{align*}
  \mathcal X%
  &=\frac{\partial}{\partial t}+v\frac{\partial}{\partial x}+X\frac{\partial}{\partial v},
 \end{align*}
 and let
 \[F=L_{vv}^{-1}\frac{\partial f}{\partial v},\]
 so that
 \[d_v(L_{vv}F)=d_v(d_vf)=0.\]
 Note that since $C$ is a second derivative of $\mu$, we have
 \[\int \left(v\frac{\partial}{\partial x}+C\frac{\partial}{\partial v}+\frac{\partial}{\partial t}\right)f\,d\mu=0.\]
 Now,
 \begin{align*}
  \int \mathcal Xf\,d\mu &=\int \left(v\frac{\partial}{\partial x}+X\frac{\partial}{\partial v}+\frac{\partial}{\partial t}\right)f\,d\mu\\
  &=\int \left(v\frac{\partial}{\partial x}+( X+C-C)\frac{\partial}{\partial v}+\frac{\partial}{\partial t}\right)f\,d\mu\\
  &=\int ( X-C)\frac{\partial f}{\partial v}\,d\mu%
  \\
  &=\int ( XL_{vv}-CL_{vv})L_{vv}^{-1}\frac{\partial f}{\partial v}\,d\mu\\
  &=\int (L_x-L_{vt}-vL_{xv}-CL_{vv})F\,d\mu\\
  &=\int L_xF-(\left(v\frac{\partial}{\partial x}+C\frac{\partial}{\partial v}+\frac{\partial}{\partial t}\right)L_v)F\,d\mu\\
  &=\int L_xF+L_v\left(v\frac{\partial}{\partial x}+C\frac{\partial}{\partial v}+\frac{\partial}{\partial t}\right)F\,d\mu\\&=\langle\eta_{F,\mu,C},L\rangle
 \end{align*}
since 
\[\int \left(v\frac{\partial}{\partial x}+C\frac{\partial}{\partial v}+\frac{\partial}{\partial t}\right)(L_vF)\,d\mu=0
\]
because $L_vF\in C^1(T\R^n\times I)$ and $C$ is a second derivative of $\mu$.

\eqref{it:horizontallycritical}$\Rightarrow$\eqref{it:invariant}. 
If $\mu$ is closed with closed velocities and horizontally critical, and $C$ is some second derivative for it, we get from our computation above that, for all $f\in C^\infty_c (T\R^n\times I)$, if we let $F=L_{vv}^{-1}\partial f/\partial v$, then
\[\int \mathcal Xf\,d\mu=\langle\eta_{F,\mu,C},L\rangle=0,\]
so $\mu$ is invariant under the Euler-Lagrange flow.

\eqref{it:invariant}$\Rightarrow$\eqref{it:horizontallycritical}. 
If $\mu$ is invariant and if $\psi\in C^\infty(\R^n\times I)$ vanishes on $\R^n\times\{0,t_0\}$, we can let $\phi(x,v,t)=\psi(x,t)$, and we then have
\[\int d\psi\,d\mu=\int v\frac{\partial \psi}{\partial x}+\frac{\partial\psi}{\partial t}\,d\mu=\int v\frac{\partial \phi}{\partial x}+X\frac{\partial\phi}{\partial v}+\frac{\partial\phi}{\partial t}\,d\mu =\int \mathcal X\phi\, d\mu=0,\]
so $\mu$ is closed. As explained in Example \ref{ex:invariantisclosedvelocity}, invariance then implies that $\mu$ is of closed velocity and that $X$ is a second derivative. Now assume that $C$ is a second derivative for $\mu$, and that $F\colon T\R^n\times I\to\R^n$ is a smooth function such that the 1-form $FL_{vv}$ is closed when restricted to the fibers of $T\R^n\times I$, that is, $d_v(FL_{vv})=0$. Because of this last property, by Poincar\'e's lemma there is a function $f\in C^\infty(T\R^n\times I)$ such that $\partial f/\partial v=L_{vv}F$. From the computation above, we get that 
\[\langle\eta_{F,\mu,C},L\rangle=\int \mathcal Xf\,d\mu=0.\]
So $\mu$ is horizontally critical.
\end{proof}

\subsection{Insufficiency of the class of velocity-independent horizontal variations}
\label{sec:insufficiency}

The question arises of whether a smaller class of variations would still characterize the invariance under the Euler-Lagrange flow than the one suggested in Section \ref{sec:horizontalcriticality}. This is an open problem. 

A natural class to ask about is the set of velocity-independent horizontal variations, defined in Example \ref{ex:horizontal}. These variations were the ones the Euler-Lagrange equations were originally deduced from. This class turns out to be insufficient for a characterization of invariance under the Euler-Lagrange flow, as the following example shows.

\begin{example}
 This example is inspired in an idea of Alberto Abbondandolo.
 For simplicity we will work in the time-independent setting, although the example can be easily extended to the time-dependent one. 
 
 Let $\T=\R/\Z$ be the 1-dimensional torus with tangent bundle. Let $\Theta_\T$ be the set of distributions corresponding to velocity-independent horizontal variations; these can be described explicitly: given a $C^1$ function $\beta\colon\T\to\R$ and a measure $\mu\in\mathscr H_\Delta$, the distribution $\eta_{\beta,\mu}$ acts on $\phi\in C^\infty(T\T)$ like this:
\[\langle\eta_{\beta,\mu},\phi\rangle=\int_0^1 \beta(x)\frac{\partial \phi}{\partial x}(x,v)+\beta'(x)\frac{\partial \phi}{\partial v}(x,v)\,d\mu.\]

Let us construct an example of a measure $\mu\in \mathscr H_\Delta$ that is $\Theta_\T$-critical but not invariant under the Euler-Lagrange flow for the Lagrangian $L(x,v)=v^2$. Let $k\geq 4$ and, for $i=1,2,\dots,k$, let $\rho_i\colon\R\to\R_{\geq0}$ be nonnegative, $C^\infty$ functions supported on $\R_{\geq 0}$, such that the interior of $\supp \rho_i\cap\supp\rho_j$ is nonempty for all $1\leq i,j\leq k$, $\int_\R\rho_i(v)dv=1$, and all the numbers 
\[r_i=\int_\R v\rho_i(v)\,dv\quad\textrm{and}\quad s_i=\int_\R v^2\rho_i(v)\,dv\]
are positive and distinct. 
Let
\[\mu=\sum_{i=1}^k a_i(x)\,\rho_i(v)\,dv\,dx\]
for some functions $a_i\colon\T\to\R_{>0}$ whose properties we will now determine. Here, $dx$ and $dv$ stand for Lebesgue measure on $\T$ and $T_x\T\cong \R$, respectively. Note that $\mu$ is indeed in $\mathscr H_\Delta$ by Example \ref{ex:smoothisclosedvelocity}.
It can be checked that
\begin{enumerate}
 \item\label{it:aaprobability} the measure $\mu$ is a probability if, and only if, $\sum_i\int_0^1 a_i(x)\,dx=1$,
 \item\label{it:aaclosed} the measure $\mu$ is closed if, and only if, the function $x\mapsto\sum_ir_ia_i(x)$ is constant,
 \item\label{it:aacritical} the measure $\mu$ is $\Theta_\T$-critical if, and only if, the function $x\mapsto \sum_is_ia_i(x)$ is constant, and
 \item\label{it:aainvariant} the measure $\mu$ is invariant under the Euler-Lagrange flow for $L$ if, and only if, the functions $a_i$ are constant.
\end{enumerate}
It can be seen that, since $k\geq 4$, there exist smooth functions $a_i$ that satisfy items \ref{it:aaprobability}--\ref{it:aacritical} while not satisfying item \ref{it:aainvariant}.
\end{example}

\subsection{Ma\~n\'e's theorem}
\label{sec:manethm}

Let $L\colon T\R^n\times I\to\R$ be a $C^2$ function, and assume that 
there is a constant $c\in\R$ such that $L+c$ is a $\mu$-minimizable Lagrangian, so that there is a Lipschitz function $f$ such that $L+c\geq df$ with equality throughout $\supp\mu$, with the additional property that this equality occurs on exactly one point on each fiber over the projected support $\pi_{\R^n\times I}(\supp\mu)$.

\begin{coro}
 In the setting described above, the minimizer $\mu$ must be invariant under the Euler-Lagrange flow.
\end{coro}

\begin{rmk}
A version of this corollary was proved by Ma\~n\'e \cite[\S2-4]{contrerasiturriagabook}, assuming that the Lagrangian $L$ is convex and superlinear (also known as \emph{Tonelli}, which immediately implies the existence of minimizers, whence by Theorem \ref{thm:mathertimedependent} that setting is contained in the one described above). The weak \textsc{kam} theory \cite{fathibook,fathisiconolfi04,fathisiconolfi05} gives alternative tools that also prove Ma\~n\'e's result.
\end{rmk}

\begin{proof}
 It follows from item \ref{it:fisdifferentiable} in Theorem \ref{thm:mathertimedependent} and Lemma \ref{lem:thatextraregularity} that every minimizer $\mu\in\mathscr H$ of the action $A_L(\nu)=\int L\,d\nu$ has the property that its support is contained in the graph of a Lipschitz section $\sigma=df$ of the bundle $T\R^n\times I\to\R^n\times I$ that is differentiable on $\pi_{\R^n\times I}(\supp\mu)$. As a consequence, $\mu$ is of closed velocity, $\mu\in\mathscr H_\Delta$, because the derivative of $\sigma$ can be used as a second derivative for $\mu$; cf. Example \ref{ex:smoothisclosedvelocity}. 
 The statement of the corollary follows immediately from Theorem \ref{thm:invariance}.
\end{proof}

\appendix

\section{Previous results}
\label{sec:previous}
\subsection{Differentiable curves in a convex set}
\label{sec:differentiablecurves}

Recall that the \emph{tangent cone} of a convex set $C$ at a point $p$ is the set $\R_{\geq0}(C-p)$.

\begin{thm}[{\cite[Corollary 3]{mydifferentiablecurves}}]
 \label{thm:differentiablecurves}
 Let $V$ be a locally-convex topological vector space, $p$ an element of a convex subset $C$ of $V$, and $v$ an element of $V$. Assume that $v$ and $C$ are contained in a subset $X$ of $V$ such that there is a norm on $X$ that induces the topology that $X$ inherits from $V$. Assume additionally that $X$ contains the intersection of the tangent cone of $C$ at $p$ with an open subset of $V$ that contains $v$. Then there is a curve $c\colon[0,+\infty)\to C$ such that 
 \[c(0)=p\quad\textrm{and}\quad dc(t)/dt|_{t=0+}=v\]
 if, and only if, $\theta(v)\geq 0$ for all $\theta$ in $V^*$ for which $\theta(p)=\inf_{q\in C}\theta(q)$, that is, for all continuous linear functionals $\theta$ that attain their minimum within $C$ at $p$.
\end{thm}

\begin{lem}[{\cite[Lemma 7]{mydifferentiablecurves}}]\label{lem:boundednormability}
 Let $k\geq 1$ and $n\geq 0$ be two integers, $\gamma>0$, and $U$ a bounded open subset of $\R^k$, so that its closure $\overline W$ is compact. Consider the subset $X$ of $\mathscr E'$ consisting of compactly supported distributions $\xi$ of degree at most $n$ with support in $W$ and such that $|\xi|'_{\overline W,n}\leq \gamma$. Then the seminorm $|\cdot|'_{\overline W,n}$ is actually nondegenerate on $X$, and within $X$ it acts as a norm that induces the topology that $X$ inherits from $\mathscr E'$, that is, the weak* topology with respect to $C^\infty(\R^k)$. In particular, the sequential closure of subsets of $X$ coincides with their closure.
\end{lem}

\subsection{Minimizable Lagrangian action functionals}
\label{sec:minimizablefunctionals}

\begin{defn}
 Let $x$ be a point in an open set $U\subseteq M$.
 The form $\theta\in T_x^*M$ is a \emph{Clarke differential of $f\colon U\to \R$ at $x$} if it is in the convex hull of the accumulation points of the values of $df_y$ as $x\to y$. A section $\alpha\colon U\to T^*U$ is a \emph{Clarke differential of $f\colon U\to\R$} if it is a Clarke differential at every point of $U$.
\end{defn}

Let $\mathscr H$ be as defined in Section \ref{sec:variations}.

\begin{thm}[{\cite[Corollary 10]{myminimizable}}] \label{thm:mathertimedependent}
 Assume that $L$ is an element of $C^2(T\R^n\times I)$ such that $\nu\mapsto\int_{T\R^n\times I}L\,d\nu$ reaches its minimum within $\mathscr H$ at some point $\mu$. Let $U$ be a bounded and open subset of $\R^n$ such that $\pi(\supp\mu)\subseteq U\times I$. Let $\sigma\colon T\R^n\times I\to T(\R^n\times I)$ be given by $\sigma(x,v,t)=(x,v,t,\mathbf 1)$. Then there exist a Lipschitz function $f\colon U\times I\to\R$, a nonnegative function $g\colon TU\times I\to \R_{\geq 0}$, and a bounded (possibly discontinuous) section $\alpha\colon U\times I\to T^*(U\times I)$ such that:
 \begin{enumerate}
  \item $\alpha$ is a Clarke differential of $f$ (in particular, $df=\alpha$ wherever $f$ is differentiable;
  \item\label{it:decompositionL} throughout $TU\times I$ we have
   \[L=\frac{1}{\mu(T\R^n\times I)}\int_{T\R^n\times I}L\,d\mu+\alpha\circ \sigma+g,\]
  which in particular means that $L=\int_{T\R^n\times I}L\,d\mu/\mu(T\R^n\times I)+df\circ \sigma+g$ wherever $f$ is differentiable;
  \item\label{it:fisdifferentiable} for every open set $V\subset \R^n\times I$ not intersecting the boundary $\R^n\times\{0,t_0\}$, $f$ is differentiable on $\pi(\supp\mu)\cap V$ and, for $(x_0,t_0,v_0,\mathbf 1)\in\supp\mu\cap TV$, we have
   \begin{multline*}
    df_{(x_0,t_0)}(v,\tau)=\frac{\partial L}{\partial v}(x_0,t_0,v_0)\cdot v+\\\frac{1}{\mu(T\R^n\times I)}\left(L(x_0,t_0,v_0)-\frac{\partial L}{\partial v}(x_0,t_0,v_0)\cdot v_0-\int_{T\R^n\times I}L\,d\mu\right)\cdot \tau, 
   \end{multline*}
   where $v\in T_{x_0}\R^n$, $\tau\in TI$, and $\partial L/\partial v$ denotes the derivative of the restriction of $L$ to $T_{x_0}\R^n\times \{t\}\times \{1\}$, $t\in I$;
  \item\label{it:lipschitzity} for every open neighborhood $Y\subset U$ of $U\times \{0,t_0\}$, there is a constant $C_Y>0$ such that the map $(x,t)\mapsto df_{(x,t)}$ is $C_Y$-Lipschitz throughout $\pi(\supp\mu)\cap U\setminus Y$;
  \item $g\equiv 0$ throughout $\supp\mu$;
 \end{enumerate}
\end{thm}

 \section{The Euler--Lagrange equations for minimizers}
\label{sec:eulerlagrange}
Here we want to discuss the extent to which Theorem \ref{thm:mathertimedependent} implies that minimizers obey the Euler-Lagrange equations,
\begin{equation}\label{eq:eulerlagrangeappx}
\frac{d}{dt}\frac{\partial L}{\partial v}(\gamma(t),\gamma'(t),t)=
 \frac{\partial L}{\partial x}(\gamma(t),\gamma'(t),t),
\end{equation}
where $L$ is a function in $C^2(T\R^n\times \R)$ whose action reaches its minimum at a measure $\mu\in \mathscr H_1(c)$. 

We first want to remark that there are many circumstances in which the minimizers have no regularity whatsoever (see for example \cite[Example 23]{myminimizable}), and in general it is unrealistic to expect anything like \eqref{eq:eulerlagrangeappx} to hold.

So let us instead assume that the support of a minimizing measure $\mu$ contains the graph of a twice-differentiable curve $\gamma\colon(-\varepsilon,\varepsilon)\to \R^n$, $\varepsilon>0$. By item \ref{it:fisdifferentiable} in Theorem \ref{thm:mathertimedependent}, and with the same notations as in the theorem, we have that
\[\frac {\partial L}{\partial v}(\gamma(t),\gamma'(t),t)=\frac{\partial f}{\partial x}(\gamma(t),t).
\]
Also, by item \ref{it:decompositionL} in that theorem, we have that the following quantity is constant throughtout the support of $\mu$:
\[L-\frac{\partial f}{\partial x}(x,t)v-\frac{\partial f}{\partial t}(x,t).\]
It follows that
\begin{multline*}
 \frac{d}{dt}\frac{\partial L}{\partial v}(\gamma(t),\gamma'(t),t)=\frac{d}{dt}\frac{\partial f}{\partial x}(\gamma(t),t)
 =\frac{\partial^2f}{\partial x^2}(\gamma(t),t)\gamma'(t)+\frac{\partial^2 f}{\partial x\partial t}(\gamma(t),t)\\
 =\left.\frac{\partial}{\partial x}\left(\frac{\partial f}{\partial x}\cdot v+\frac{\partial f}{\partial t}\right)\right|_{(x,v,t)=(\gamma(t),\gamma'(t),t)}
 =\frac{\partial L}{\partial x}(\gamma(t),\gamma'(t),t).
\end{multline*}
Thus indeed $\gamma$ satisfies the Euler-Lagrange equations \eqref{eq:eulerlagrangeappx}.

We conclude that, in a sense, the existence of the function $f$ as in Theorem \ref{thm:mathertimedependent} (which one may call, a \emph{critical subsolution of the Hamilton-Jacobi equation}; cf. \cite{fathibook}) is more fundamental than the Euler-Lagrange equations, as the latter can be deduced from the former, and the latter may not always hold, while the statement of the theorem shows that the former always does. 

\bibliography{bib}{}

\begin{thebibliography}{10}

\bibitem{bangert}
V.~Bangert.
\newblock Minimal measures and minimizing closed normal one-currents.
\newblock {\em Geom. Funct. Anal.}, 9(3):413--427, 1999.

\bibitem{patrick}
Patrick Bernard.
\newblock Young measures, superposition and transport.
\newblock {\em Indiana Univ. Math. J.}, 57(1):247--275, 2008.

\bibitem{contrerasiturriagabook}
Gonzalo Contreras and Renato Iturriaga.
\newblock {\em Global minimizers of autonomous {L}agrangians}.
\newblock 22$^{\rm o}$ Col\'oquio Brasileiro de Matem\'atica. [22nd Brazilian
  Mathematics Colloquium]. Instituto de Matem\'atica Pura e Aplicada (IMPA),
  Rio de Janeiro, 1999.

\bibitem{fathibook}
Albert Fathi.
\newblock Weak {KAM} theorem in lagrangian dynamics.
\newblock Preliminary Version Number 10, June 2008.

\bibitem{fathisiconolfi04}
Albert Fathi and Antonio Siconolfi.
\newblock Existence of {$C^1$} critical subsolutions of the {H}amilton-{J}acobi
  equation.
\newblock {\em Invent. Math.}, 155(2):363--388, 2004.

\bibitem{fathisiconolfi05}
Albert Fathi and Antonio Siconolfi.
\newblock P{DE} aspects of {A}ubry-{M}ather theory for quasiconvex
  {H}amiltonians.
\newblock {\em Calc. Var. Partial Differential Equations}, 22(2):185--228,
  2005.

\bibitem{federerrealflat}
Herbert Federer.
\newblock Real flat chains, cochains and variational problems.
\newblock {\em Indiana Univ. Math. J.}, 24:351--407, 1974/75.

\bibitem{giaquintamodicasoucek2}
Mariano Giaquinta, Giuseppe Modica, and Ji{\v{r}}{\'{\i}} Sou{\v{c}}ek.
\newblock {\em Cartesian currents in the calculus of variations. {II}},
  volume~38 of {\em Ergebnisse der Mathematik und ihrer Grenzgebiete. 3. Folge.
  A Series of Modern Surveys in Mathematics [Results in Mathematics and Related
  Areas. 3rd Series. A Series of Modern Surveys in Mathematics]}.
\newblock Springer-Verlag, Berlin, 1998.
\newblock Variational integrals.

\bibitem{giusti}
Enrico Giusti.
\newblock {\em Direct methods in the calculus of variations}.
\newblock World Scientific Publishing Co., Inc., River Edge, NJ, 2003.

\bibitem{matheractionminimizing91}
John~N. Mather.
\newblock Action minimizing invariant measures for positive definite
  {L}agrangian systems.
\newblock {\em Math. Z.}, 207(2):169--207, 1991.

\bibitem{myminimizable}
Rodolfo Rios-Zertuche.
\newblock Characterization of minimizable {L}agrangian action functionals and a
  dual {M}ather theorem.
\newblock Preprint. arXiv:1810.03433 [math.OC].

\bibitem{mydifferentiablecurves}
Rodolfo Rios-Zertuche.
\newblock Existence of differentiable curves in convex sets and the concept of
  direction of the flow in mass transportation.
\newblock Preprint. arXiv:1810.05999 [math.OC].

\bibitem{wuertz}
Michael Wuertz.
\newblock {\em The implicit function theorem for {L}ipschitz functions and
  applications}.
\newblock PhD thesis, University of Missouri-Columbia, 2008.

\end{thebibliography}
\bibliographystyle{plain}
\end{document}